\documentclass[12pt]{amsart}
\usepackage[paperwidth=8.5in,paperheight=13in, margin=1in]{geometry}
\newtheorem{lemma}{Lemma}[section]
\newcommand{\RNum}[1]{\uppercase\expandafter{\romannumeral #1\relax}}
\usepackage{setspace}
\usepackage{amsmath}
\usepackage{amsthm}

\newtheorem{theorem}{Theorem}[section]
\newtheorem{remark}{Remark}[section]

\usepackage{pdfpages}
\usepackage{amssymb,latexsym}
\newcommand{\beq}{\begin{equation}}
\newcommand{\eeq}{\end{equation}}

\newtheorem{property}{Property}

\newtheorem{thm}{Theorem}[section]

\newtheorem{lem}[thm]{Lemma}

\begin{document}
A Second Type Of Higher Order Generalised Geometric Polynomials and Higher Order Generalised Euler Polynomials.

\vspace{1cm} 

\tiny

\noindent Sithembele Nkonkobe$^a$, Levent Kargin$^b$ , Bayram \c{C}ekim$^c$, Roberto B. Corcino$^d$, Cristina B. Corcino$^e$
\vspace{5mm}

\tiny

\noindent $^a$ Department of Mathematical Sciences, Sol Plaatje University, Kimberley, 8301, South Africa,  snkonkobe@gmail.com

\noindent$^b$ Akdeniz University, Faculty of Science, Department of Mathematics, Antalya, Turkey, lkargin@akdeniz.edu.tr

\noindent $^c$ Gazi University, Faculty of Science, Department of Mathematics, Teknik okullar, Ankara, Turkey,\\\indent bayramcekim@gazi.edu.tr
 
 \noindent$^d$ Research Institute for Computational Mathematics and Physics, Cebu Normal University,  Cebu City, Phillipines,\\\indent 600, rcorcino@yahoo.com
 
 \noindent$^e$ Mathematics Department, Cebu Normal University, Cebu City, Phillipines, 6000, cristinacorcino@yahoo.com\normalsize
\section*{Abstract}

In this study we introduce a second type of higher order generalised geometric polynomials. 
This we achieve by examining the generalised stirling numbers $S(n,k,\alpha,\beta,\gamma)$ [Hsu \& Shiue,1998] for some negative arguments.
 We study their number theoretic properties,  asymptotic properties, and study their combinatorial properties using the notion of barred preferential arrangements. We also proposed a generalisation of the classical Euler polynomials and show how these Euler polynomials are related to the second type of higher order generalised geometric polynomials.

\noindent\fontsize{10}{1}\selectfont     
Mathematics Subject Classifications : 05A15, 05A16, 05A18, 05A19,   11B73, 11B83
\\\textbf{Keyword(s)}:Preferential arrangement, barred preferential arrangement, geometric polynomial, Euler polynomials.
\normalsize
\vspace{2mm}

\section{Introduction}

A \textit{barred preferential arrangement}(see \cite{barred:2013,Pipenger hypercube}) is formed when a number of identical bars are inserted in-between the blocks of a given preferential
arrangement of $X_n$($n$-element set). For instance, consider the following barred preferential arrangements(BPAs) of $X_8$;

\RNum{1}). $2\quad58\quad4|\quad|13\quad6|7$,

\RNum{2}). $|\quad|578\quad 3\quad1|2\quad146|$.

The BPA in \RNum{1} has three bars, hence four sections. The first section has three blocks formed by the elements $\{2\}$, $\{5,8\}$, $\{4\}$.
The second section is empty i.e the section between the first two bars(from left to right). The third section has two blocks formed by the
elements $\{1,3\}$, and $\{6\}$. The fourth section has only one block $\{7\}$.

The BPA in \RNum{2} has four bars hence five sections are formed. The first two sections are empty. The third section has three blocks. The
fourth section has two blocks. The fifth section is empty.

In \cite{Chains in power set Nelsen 91} Nelsen and Schmidt proposed the generating function,
\begin{equation}\label{equation:5}
\frac{e^{\gamma t}}{2-e^t}.
\end{equation}

The generating function for $\gamma=0$ is known to be that of the sequence of number of preferential arrangements (see
\cite{gross:1962,mendelson:1982}).
 In the manuscript Nelsen and Schmidt for $\gamma=2$ interpreted the generating function as that of the number of chains in the power set of an
 $n$-element set $X_n$. They then asked ``could there be other combinatorial structure associated either with $X_n$ or the power set of $X_n$
 whose integer sequences are generated by the generating function in \eqref{equation:5} for other values of $\gamma$?". We will refer to this
 question as the Nelsen-Schmidt question, and to the generating function in \eqref{equation:5} as the Nelsen-Schmidt generating function.

 In answering the Nelsen-Schmidt question combinatorial interpretations of generating functions which are generalising the Nelsen-Schmidt
 generating function exist in the literature (see \cite{Nkonkobe & Murali Nelesn-Schmidt generating function,A comb analysis of geo polynomials,Our paper on Generalised barred preferential arrangements,Muralis paper on the Nelsen-Schmidt generating function}). In this manuscript we
 propose an alternative  generating function\\ $(1-\alpha t)^{\frac{-\gamma}{\alpha}}\begin{bmatrix}
 \frac{1}{1-x[(1-\alpha t)^{\frac{-\beta}{\alpha}}-1]}
 \end{bmatrix}^\lambda$, which also generalises that of Nelsen and Schmidt, and give combinatorial interpretation of the integer sequences that
 arise from the generating function. This is analogous to remark 3.2 of \cite{A comb analysis of geo polynomials}.

Geometric polynomials are well known in the literature\cite{GeoP1,GeoP2,GeoP3,GeoP4,GeoP5,GeoP6,Nkonkobe & Murali Nelesn-Schmidt generating function}. A generalisation of geometric polynomials is higher order generalised geometric polynomials which seem to first appear in \cite{Higher order geometric polynomials}, and their generating function is $(1+\alpha t)^{\frac{\gamma}{\alpha}}\begin{bmatrix}
 \frac{1}{1-x[(1+\alpha t)^{\frac{\beta}{\alpha}}-1]}
 \end{bmatrix}^\lambda$.  The polynomials have been extensively studied in both  \cite{A comb analysis of geo polynomials,paper on recurrence relation on H G polynomials by Levent}.

In this study we propose a new type of higher order generalised geometric polynomials by proposing the following generating function  
 $(1-\alpha t)^{\frac{-\gamma}{\alpha}}\begin{bmatrix}
  \frac{1}{1-x[(1-\alpha t)^{\frac{-\beta}{\alpha}}-1]}
  \end{bmatrix}^\lambda$.
 
 The classical Euler polynomials and some of their generalisations are well known in the literature for instance in \cite{EulerP1,EulerP2,EulerP3,EulerP4,EulerP5,EulerP6,EulerP7,EulerP8,EulerP9,EulerP10}. In this study we define a new generalisation of  Euler polynomials and show how these new Euler polynomials are related to the new type of higher order generalised geometric polynomials.  

A generalisation $S(n,k,\alpha,\beta,\gamma)$ of the classical stirling numbers  is given in \cite{A unified approach to generalized Stirling numbers} in the following way, $(t|\alpha)_n=\sum\limits_{k=0}^{n}S(n,k,\alpha,\beta,\gamma)(t-\gamma|\beta)_n$, where $(t|\alpha)_n$ is the generalised factorial polynomial $(t|\alpha)_n=\prod\limits_{k=0}^{n-1}(t-k\alpha)$, where $n\geq1$ and $\alpha,\beta,\gamma$ real or complex not all equal to zero. For  both the higher order generalised geometric polynomials and the higher order generalised Euler polynomials we make use of these stirling numbers.

\section{Combinatorial Interpretations.}
\begin{property}\cite{Paper with interpretation of G-stirling numbers}
\label{property:1}Given $n$ elements and a single cell with $\gamma$ compartments such that $\alpha|\gamma$, such that when an element lands on
one of the compartments, the compartment splits into $\alpha+1$ compartment. The total number of ways of distributing $n$ elements in the
$\gamma$ compartments is $(\gamma|-\alpha)_n$.
\end{property}
\begin{lemma}\label{lemma:2}\cite{Paper with interpretation of G-stirling numbers}
Suppose $\alpha,\beta,\gamma$ are non-negative integers such that $\alpha$ divides both $\beta$ and $\gamma$. Assume there are $k+1$ cells,
where each of the first $k$ cells has $\beta$ compartments, and the $(k+1)^{th}$ cell has $\gamma$ compartments. The compartments have the
property that when an elements lands on a compartment, the compartment splits into $\alpha+1$ compartments. The number of ways of distributing
$n$ elements into $k+1$ cells is $(-1)^{n+k}\beta^kk!S(n,k,\alpha,-\beta,-\gamma)$.
\end{lemma}

We generalise Lemma~\ref{lemma:2} in the following way without considering the cell with $\gamma$ compartments.
\begin{property}\label{propoerty:2}\cite{Paper with interpretation of G-stirling numbers}
\label{lemma:1}
Given $\alpha,\beta$ (non-negative integers) such that $\alpha|\beta$. Given $k$ distinct cells such that the $k$ cells each contains $\beta$
compartments. The compartments are given cyclic ordered numbering. Each compartment has unlimited capacity. After an element lands on a
compartment, the compartment splits into $\alpha+1$ compartments. The number of all those arrangements such that none of the compartments are
empty (where $k$ runs from 0 to $n$), and each of the $k$ cells is colored by one of $x$ available colors is,
$\sum\limits_{k=0}^{n}(-1)^{n+k}\beta^kk!S(n,k,\alpha,-\beta,0)x^k$.\\
 Denote, $M^{\lambda,x}_n(\alpha,\beta,0)=\sum\limits_{k=0}^{n}(-1)^{n+k}\beta^kk!S(n,k,\alpha,-\beta,0)x^k$.\end{property}

\begin{equation}\label{equation:1}(-1)^{n+k}\beta^kk!S(n,k,\alpha,\beta,\gamma)=\sum\limits_{i=0}^n(-1)^{k-i}\binom{k}{i}(\beta
i+\gamma|-\alpha)_n.\end{equation}

 The authors in \cite{A unified approach to generalized Stirling numbers} have proposed the following,
 \begin{equation}\label{equation:2}
 (1+\alpha t)^{\frac{\gamma}{\alpha}}\begin{bmatrix}
 \frac{(1+\alpha t)^{\frac{\beta}{\alpha}}-1}{\beta}
 \end{bmatrix}^k=k!\sum\limits_{n=0}^{\infty}S(n,k,\alpha,\beta,\gamma)\frac{t^n}{n!}.
 \end{equation}

Let $A^{\lambda,x}_n(\alpha,\beta,\gamma)=\sum\limits_{k=0}^{n}\binom{k+\lambda-1}{k}(-1)^{n+k}\beta^kk!S(n,k,\alpha,-\beta,-\gamma)x^k$.

Its follows from \eqref{equation:2} that,

\begin{equation}\label{equation:11}
(1-\alpha t)^{\frac{-\gamma}{\alpha}}\begin{bmatrix}
\frac{1}{1-x[(1-\alpha t)^{\frac{-\beta}{\alpha}}-1]}
\end{bmatrix}^\lambda=\sum\limits_{n=0}^{\infty}A^{\lambda,x}_n(\alpha,\beta,\gamma)\frac{t^n}{n!}.
\end{equation}
\begin{theorem} \label{theorem:1}For $\alpha,\beta,\gamma,x\in\mathbb{N}_0$ such that $(\alpha,\beta,\gamma,x)\not=(0,0,0,0)$,

 $A^{\lambda,x}_n(\alpha,\beta,\gamma)$
 is the number of barred preferential arrangements having $\lambda$ bars such that one section has property~\ref{property:1} and $\lambda$
 sections  having property~\ref{propoerty:2}.\end{theorem}
 \begin{proof}
 By \eqref{equation:11} we\fontsize{9}{1} have,\begin{equation}
 A^{\lambda,x}_n(\alpha,\beta,\gamma)=\sum\limits_{n_1+n_2+\cdots+
 n_{\lambda+1}=n}\binom{n}{n_1,n_2,\ldots,n_{\lambda+1}}(\gamma|-\alpha)_{n_1}\prod_{k=2}^{\lambda+1}M^{\lambda,x}_{n_k}(\alpha,\beta,0).
 \end{equation}
 \end{proof}


\begin{remark}[Nelsen-Schmidt generating Function]
Theorem~\ref{theorem:1} offers a generalised answer to the Nelsen-Schmidt question; as the generating function of Nelsen and Schmidt in
\eqref{equation:5} arises as a special case of the generating function, $(1-\alpha t)^{\frac{-\gamma}{\alpha}}\begin{bmatrix}
\frac{1}{1-x[(1-\alpha t)^{\frac{-\beta}{\alpha}}-1]}
\end{bmatrix}^\lambda$.
\end{remark}

Theorems \ref{theorem:6},\ref{theorem:2}, and \ref{theorem:4} are analogous to theorems 3.1, 3.7, 3.5 of \cite{A comb analysis of geo polynomials}.
\begin{theorem}\label{theorem:6}For $\alpha,\beta,\gamma,x\in\mathbb{N}_0$ such that $(\alpha,\beta,\gamma,x)\not=(0,0,0,0)$,
\begin{equation}
A^{\lambda,x}_{n+1}(\alpha,\beta,\gamma)=\gamma A^{\lambda,x}_n(\alpha,\beta,\gamma+\alpha)+x\lambda\beta
A^{\lambda+1,x}_n(\alpha,\beta,\gamma+\beta+\alpha).
\end{equation}
\end{theorem}
\begin{proof}
The proof is based on the position of the $(n+1)th$ element. There are two cases to consider.

When the $(n+1)th$ element is on the section with property~\ref{property:1}. There are $\gamma$ compartments to choose from in-order to place
the $(n+1)th$ elements. Placing the elements with create $\alpha$ more compartments. The remainder $n$ elements may be placed on the $\lambda+1$
sections in $A^{\lambda,x}_n(\alpha,\beta,\gamma+\alpha)$.

In the other case, the $(n+1)th$ element is in one of the sections with property~\ref{propoerty:2}. There are $\lambda$ ways of choosing a
section. Say the $(n+1)th$ element is on the $j^{th}$ section. Suppose the $(n+1)th$ element is in cell $P$ within the $j^{th}$ section. A
compartment to place the $(n+1)th$ element can be chosen in $\beta$ ways.  The cell $P$ can be colored in $x$ ways.  The portion of the $j^{th}$
cell to the left of $P$ excluding $P$, gives rise to a single section having property~\ref{propoerty:2}, also the portion of the $j^{th}$
section to the right of $P$ excluding $P$, gives rise to a single section having property~\ref{propoerty:2}. Also treat the single section
having property~\ref{property:1} and $P$ as a single unit having $\gamma+\beta+\alpha$ compartments. The remainder $n$ elements can be arranged
in $A^{\lambda+1,x}_n(\alpha,\beta,\gamma+\beta+\alpha)$.
\end{proof}

\begin{theorem}\label{theorem:2}For $\alpha,\beta,\gamma,x\in\mathbb{N}_0$ such that $(\alpha,\beta,\gamma,x)\not=(0,0,0,0)$,\fontsize{10}{1}
\begin{equation}
A^{\lambda,x}_{n+1}(\alpha,\beta,\gamma)=\gamma
A^{\lambda,x}_{n}(\alpha,\beta,\gamma+\alpha)+\sum\limits_{k=0}^{n}\binom{n}{k}A^{0,x}_k(\alpha,\beta,\gamma)A_{n-k+1}^{\lambda,x}(\alpha,\beta,0).
\end{equation}

\end{theorem}
\begin{proof}
The proof of this theorem is also based on the position of the $(n+1)th$ element.

As before when the $(n+1)th$ element is in the section having $\gamma$ compartments the elements can be arranged in $\gamma
A^{\lambda,x}_n(\alpha,\beta,\gamma+\alpha)$ ways.

In the other case we choose $k$ elements which are to go into the section having property~\ref{property:1}. The $k$ elements can be arranged in
$A^{\lambda,x}_k(\alpha,0,\gamma)$. The remainder elements can be arranged on the $\lambda$ other sections together with the $(n+1)th$ element
in $A^{\lambda,x}_{n-k+1}(\alpha,\beta,0)$ ways.
\end{proof}
\begin{theorem}\label{theorem:4}For $\alpha,\beta,\gamma,x\in\mathbb{N}_0$ such that $(\alpha,\beta,\gamma,x)\not=(0,0,0,0)$,\fontsize{9.5}{1}
\begin{equation}\label{equation:3}
A^{\lambda,x}_{n+1}(\alpha,\beta,\gamma)=\gamma
A^{\lambda,x}_n(\alpha,\beta,\gamma+\alpha)+x\lambda\beta\sum\limits_{k=0}^{n}\binom{n}{k}A^{1,x}_{k}(\alpha,\beta,\gamma+\beta+\alpha)A^{\lambda,x}_{n-k}(\alpha,\beta,0).
\end{equation}
\end{theorem}
\begin{proof}
The case when the $(n+1)th$ element is on the section having property~\ref{property:1} is as before.

Now suppose the $(n+1)th$ element is in one of the sections having property~\ref{propoerty:2}. The section, cell, compartment, and color of the
cell for the $(n+1)th$  element can all be chosen in $x\lambda\beta$ ways. Say the $(n+1)th$ element is on the $i^{th}$ section between the bars
$|^1$ and $|^2$ section. Denote the cell that the $(n+1)th$ element is on by $P$. The cell $P$ together with the single cell on the section
having property~\ref{property:1} can be thought of as a single unit having $\gamma+\beta+\alpha$ compartments, name it $M$. The portion of the
$i^{th}$ section from $|^1$ to $P$ excluding $P$ gives rise to a section having property~\ref{propoerty:2}.  Similarly the portion of the
$i^{th}$ section from $P$ (excluding $P$) to $|^2$ gives rise to a section having property~\ref{propoerty:2}, denote this second section by $Q$.
Off the $n$ elements (excluding the $(n+1)th$ element), we choose $k$ of them. The $k$ elements can occupy $M\cup Q$ in
$A^{1,x}_n(\alpha,\beta,\gamma+\beta+\alpha)$ ways. The remaining $n-k$ elements can be arranged on the other sections in
$A^{\lambda,x}_n(\alpha,\beta,0)$ ways.
\end{proof}
\begin{theorem}\label{theorem:5}For $\alpha,\beta,\gamma,x\in\mathbb{N}_0$ such that $(\alpha,\beta,\gamma,x)\not=(0,0,0,0)$,
\fontsize{11}{1}
\begin{equation}\label{equation:9}
 A^{\lambda,x}_n(\alpha,\beta,\gamma)=\sum\limits_{k=0}^{n}\binom{k+\lambda-1}{k}(-1)^{n+k}\beta^kk!S(n,k,\alpha,-\beta,-\gamma)x^k.
\end{equation}
\end{theorem}
\begin{proof}
 Given $n$ elements can be partitioned into $k+1$ cells, where $k$ have $\beta$ compartments, and one of the cells has $\gamma$ compartment in,
 $(-1)^{n+k}\beta^kk!S(n,k,\alpha,-\beta,-\gamma)$ ways, where on both types of compartments when an element lands on the compartment, the
 compartment  splits into $\alpha+1$ compartments. The $k$ cells having $\beta$ compartments can be colored in $x^k$ ways. Now, $\lambda-1$ bars
 can be inserted in-between the cells having $\beta$ compartments in $\binom{k+\lambda-1}{k}$.  \end{proof}

Thereom~\ref{theorem:5} is a generalisation of Theorem 3 of \cite{barred:2013}.

\begin{theorem}For $\alpha,\beta,\gamma,x\in\mathbb{N}_0$ such that $(\alpha,\beta,\gamma,x)\not=(0,0,0,0)$,
\begin{equation}
\label{equation:6}
A^{\lambda,x}_n(\alpha,\beta,0)=\sum\limits_{k=0}^{n}\binom{n}{k}A^{\lambda,x}_{n-k}(\alpha,\beta,\gamma)A^{0,x}_{k}(\alpha,0,\gamma)(-1)^{k}.
\end{equation}
\end{theorem}
\begin{proof}
We can have at least $k$ elements in the special  section, that has $\gamma$ compartments in
$\binom{n}{k}A^{0,x}_k(\alpha,0,\gamma)A^{\lambda,x}_{n-k}(\alpha,\beta,\gamma)$ ways. The inclusion/exclusion principle completes the proof.
\end{proof}

Combining \eqref{equation:9} and the following equation from \cite{Corcino Rorberto 2001},
\begin{equation}\label{equation:8}
S(n,k,\alpha,\gamma,\beta)=\sum_{s=k}^{n}\binom{n}{s}(\gamma|\alpha)_{n-s}S(i,k,\alpha,\beta,\gamma),
\end{equation}
we obtain \eqref{equation:7} below. In the following theorem we give a combinatorial interpretation of the result.

\begin{theorem}For $\alpha,\beta,\gamma,x\in\mathbb{N}_0$ such that $(\alpha,\beta,\gamma,x)\not=(0,0,0,0)$,
\fontsize{9}{1}
\begin{equation}\label{equation:7}
A^{\lambda,x}_n(\alpha,\beta,\gamma)=\sum\limits_{k=0}^{n}\binom{k+\lambda-1}{k}\sum\limits_{i=0}^{n}\binom{n}{i}x^k(-1)^{k+i}\beta^kk!S(i,k,\alpha,-\beta,0)(\gamma|-\alpha)_{n-i}.
\end{equation}
\end{theorem}
\begin{proof}
Given $n$ elements, $i$ of the elements can be distributed into $k$ cells all having $\beta$ compartments in
$(-1)^{k+i}\beta^kk!S(i,k,\alpha,\beta,0)$ ways. The $i$ elements can be chosen in $\binom{n}{i}$ ways. The $k$ cells can be colored in with $x$
available colors $x^k$ ways. Given $\lambda-1$ bars can be inserted in-between the $k$ cells in $\binom{k+\lambda-1}{k}$ ways. The remaining
$n-k$ elements can be distributed into a cell having $\gamma$ compartments in $(\gamma|-\alpha)_{n-k}$ ways.
\end{proof}

\section{Some Properties of the numbers $A^{\lambda,x}_{n}(\alpha,\beta,\gamma)$. }
In this section, we give some properties of $A_{n}^{\lambda,x}(\alpha
,\beta,\gamma)$ such as recurrence relations, convolution formulas and
explicit expressions. The classical Euler polynomials and some of their generalisations are well known in the literature for instance in \cite{EulerP1,EulerP2,EulerP3,EulerP4,EulerP5,EulerP6,EulerP7,EulerP8,EulerP9,EulerP10}. Moreover, we introduce a new family of polynomials which generalises the classical Euler polynomials, and 
we call \textit{higher order generalised Euler polynomials} with generalized Stirling
numbers and deal with some basic properties of these polynomials.
\\\\
The polynomial $A_{n}^{\lambda,x}(\alpha,\beta,\gamma)$ satisfy the following
recurrence relations:
\begin{theorem}
We have
\begin{equation}
xA_{n}^{\lambda+1,x}(\alpha,\beta,\gamma+\beta)=\left(  x+1\right)
A_{n}^{\lambda+1,x}(\alpha,\beta,\gamma)-A_{n}^{\lambda,x}(\alpha,\beta
,\gamma) \label{3.1}%
\end{equation}
and
\begin{equation}
A_{n+1}^{\lambda,x}(\alpha,\beta,\gamma-\alpha)-\left(  x+1\right)
\lambda\beta A_{n}^{\lambda+1,x}(\alpha,\beta,\gamma)=\left(  \gamma
-\alpha-\lambda\beta\right)  A_{n}^{\lambda,x}(\alpha,\beta,\gamma).
\label{3.2}%
\end{equation}

\end{theorem}

\begin{proof}
From (\ref{equation:11}) we have%
\[
\sum\limits_{n=0}^{\infty}\frac{d}{dx}A_{n}^{\lambda,x}(\alpha,\beta
,\gamma)\frac{t^{n}}{n!}=\lambda\left(  (1-\alpha t)^{\frac{-\beta}{\alpha}%
}-1\right)  \frac{(1-\alpha t)^{\frac{-\gamma}{\alpha}}}{\left(  1-x[(1-\alpha
t)^{\frac{-\beta}{\alpha}}-1]\right)  ^{\lambda+1}}.
\]
The right hand side of the above equation can be evaluated in two ways: In the
first way, we rewrite the right hand side as
\begin{align*}
&  \lambda\left(  (1-\alpha t)^{\frac{-\beta}{\alpha}}-1\right)
\frac{(1-\alpha t)^{\frac{-\gamma}{\alpha}}}{\left(  1-x[(1-\alpha
t)^{\frac{-\beta}{\alpha}}-1]\right)  ^{\lambda+1}}\\
&  \quad=-\frac{\lambda}{x}\frac{(1-\alpha t)^{\frac{-\gamma}{\alpha}}%
}{\left(  1-x[(1-\alpha t)^{\frac{-\beta}{\alpha}}-1]\right)  ^{\lambda}%
}+\frac{\lambda}{x}\frac{(1-\alpha t)^{\frac{-\gamma}{\alpha}}}{\left(
1-x[(1-\alpha t)^{\frac{-\beta}{\alpha}}-1]\right)  ^{\lambda+1}}.
\end{align*}
In the second way, we have
\begin{align*}
&  \lambda\left(  (1-\alpha t)^{\frac{-\beta}{\alpha}}-1\right)
\frac{(1-\alpha t)^{\frac{-\gamma}{\alpha}}}{\left(  1-x[(1-\alpha
t)^{\frac{-\beta}{\alpha}}-1]\right)  ^{\lambda+1}}\\
&  \quad=\lambda\frac{(1-\alpha t)^{\frac{-\left(  \gamma+\beta\right)
}{\alpha}}}{\left(  1-x[(1-\alpha t)^{\frac{-\beta}{\alpha}}-1]\right)
^{\lambda+1}}-\lambda\frac{(1-\alpha t)^{\frac{-\gamma}{\alpha}}}{\left(
1-x[(1-\alpha t)^{\frac{-\beta}{\alpha}}-1]\right)  ^{\lambda+1}}.
\end{align*}
Combining these two identities gives the first formula of this theorem.

For the second recurrence relation, we replace $\gamma$ with $\gamma-\alpha$
in Theorem \ref{theorem:6} and multiply (\ref{3.1}) by $\lambda\beta$ to
obtain%
\[
x\lambda\beta A_{n}^{\lambda+1,x}(\alpha,\beta,\gamma+\beta)=\left(
\gamma-\alpha\right)  A_{n}^{\lambda,x}(\alpha,\beta,\gamma)-A_{n+1}%
^{\lambda,x}(\alpha,\beta,\gamma-\alpha)
\]
and%
\[
x\lambda\beta A_{n}^{\lambda+1,x}(\alpha,\beta,\gamma+\beta)=\lambda
\beta\left(  x+1\right)  A_{n}^{\lambda+1,x}(\alpha,\beta,\gamma)-\lambda\beta
A_{n}^{\lambda,x}(\alpha,\beta,\gamma),
\]
respectively. Since the left hand side of the above equations are equal, we
arrive at (\ref{3.2}).
\end{proof}

Differentiating both sides of (\ref{equation:11}) with respect to $t$ gives%
\[
\sum\limits_{n=0}^{\infty}A_{n+1}^{\lambda,x}(\alpha,\beta,\gamma)\frac{t^{n}%
}{n!}=\gamma\frac{(1-\alpha t)^{\frac{-\left(  \gamma+\alpha\right)  }{\alpha
}}}{\left(  1-x[(1-\alpha t)^{\frac{-\beta}{\alpha}}-1]\right)  ^{\lambda}%
}+x\lambda\beta\frac{(1-\alpha t)^{\frac{-\left(  \gamma+\alpha+\beta\right)
}{\alpha}}}{\left(  1-x[(1-\alpha t)^{\frac{-\beta}{\alpha}}-1]\right)
^{\lambda+1}}.
\]
This equation yields Theorem \ref{theorem:6}. Now, we want to deal with this
equation from a different point of view. For $\gamma=\gamma_{1}+\gamma_{2}$
and $\lambda=\lambda_{1}+\lambda_{2},$ the right hand side of the above
equation can be rewriten as%
\fontsize{9}{1}
\[
\left(  \gamma_{1}+\gamma_{2}\right)  \frac{(1-\alpha t)^{\frac{-\left(
\gamma_{1}+\gamma_{2}+\alpha\right)  }{\alpha}}}{\left(  1-x[(1-\alpha
t)^{\frac{-\beta}{\alpha}}-1]\right)  ^{\lambda_{1}+\lambda_{2}}}%
+x\beta\left(  \lambda_{1}+\lambda_{2}\right)  \frac{(1-\alpha t)^{\frac
{-\left(  \alpha+\beta+\gamma_{1}\right)  }{\alpha}}}{\left(  1-x[(1-\alpha
t)^{\frac{-\beta}{\alpha}}-1]\right)  ^{\lambda_{1}}}\times\frac{(1-\alpha
t)^{\frac{-\gamma_{2}}{\alpha}}}{\left(  1-x[(1-\alpha t)^{\frac{-\beta
}{\alpha}}-1]\right)  ^{\lambda_{2}}}.
\]\normalsize
Then, we have the following convolution formula:

\begin{theorem}
\label{teo1}We have%
\begin{align*}
&  x\beta\left(  \lambda_{1}+\lambda_{2}\right)  \sum_{k=0}^{n}\binom{n}%
{k}A_{k}^{\lambda_{1},x}(\alpha,\beta,\alpha+\beta+\gamma_{1})A_{n-k}%
^{\lambda_{2},x}(\alpha,\beta,\gamma_{2})\\
&  \quad=A_{n+1}^{\lambda_{1}+\lambda_{2},x}(\alpha,\beta,\gamma_{1}%
+\gamma_{2})-\left(  \gamma_{1}+\gamma_{2}\right)  A_{n}^{\lambda_{1}%
+\lambda_{2},x}(\alpha,\beta,\gamma_{1}+\gamma_{2}+\alpha).
\end{align*}

\end{theorem}

 Another convolution formula for these polynomials is as follows:

\begin{theorem}
\label{teo2}%
\[
\sum_{k=0}^{n}\binom{n}{k}A_{k}^{\lambda_{1},x}(\alpha,\beta,\alpha
+\beta+\gamma_{1})A_{n-k}^{\lambda_{2},x}(\alpha,\beta,\gamma_{2}%
)=A_{n}^{\lambda_{1}+\lambda_{2},x}(\alpha,\beta,\alpha+\beta+\gamma
_{1}+\gamma_{2}).
\]

\end{theorem}

\begin{proof}
For $\gamma=\alpha+\beta+\gamma_{1}+\gamma_{2}$ and $\lambda=\lambda
_{1}+\lambda_{2}$ in (\ref{equation:11}), we have%
\begin{align*}\fontsize{7}{1}
\sum\limits_{n=0}^{\infty}A_{n}^{\lambda_{1}+\lambda_{2},x}(\alpha
,\beta,\alpha+\beta+\gamma_{1}&+\gamma_{2})\frac{t^{n}}{n!}   \\& =\frac
{(1-\alpha t)^{\frac{-\left(  \alpha+\beta+\gamma_{1}\right)  }{\alpha}}%
}{\left(  1-x[(1-\alpha t)^{\frac{-\beta}{\alpha}}-1]\right)  ^{\lambda_{1}}%
}\times\frac{(1-\alpha t)^{\frac{-\gamma_{2}}{\alpha}}}{\left(  1-x[(1-\alpha
t)^{\frac{-\beta}{\alpha}}-1]\right)  ^{\lambda_{2}}}\\
&  =\sum\limits_{n=0}^{\infty}\left[  \sum_{k=0}^{n}\binom{n}{k}A_{k}%
^{\lambda_{1},x}(\alpha,\beta,\alpha+\beta+\gamma_{1})A_{n-k}^{\lambda_{2}%
,x}(\alpha,\beta,\gamma_{2})\right]  \frac{t^{n}}{n!}.
\end{align*}
Comparing the coefficients of the above equation gives the desired identity.
\end{proof}

It is worth noting that Theorem \ref{teo1} and Theorem \ref{teo2} are the
generalizations of some convolution formulas given in \cite{A comb analysis of geo polynomials}.

From (\ref{equation:11}), one can obtain that%
\begin{equation}
A_{n}^{\lambda,x}(\alpha,\beta,\gamma+\beta\lambda)=A_{n}^{\lambda
,-x-1}(\alpha,-\beta,\gamma)=\left(  -1\right)  ^{n}A_{n}^{\lambda
,-x-1}(\alpha,\beta,-\gamma). \label{3.7}%
\end{equation}
Then using (\ref{equation:9}) gives the following explicit formulas for
$A_{n}^{\lambda,x}(\alpha,\beta,\gamma):$

\begin{theorem}
For all non-negative integers $n,$%
\begin{equation}
A_{n}^{\lambda,x}(\alpha,\beta,\gamma)=\left(  -1\right)  ^{n}\sum
\limits_{k=0}^{n}\binom{k+\lambda-1}{k}(-\beta)^{k}k!S_{2}(n,k,\alpha
,\beta,\beta\lambda-\gamma)\left(  x+1\right)  ^{k}, \label{3.8}%
\end{equation}

\end{theorem}

Before giving the main theorem of this section we first recall the the
generalized exponential polynomials $S_{n}\left(  x;\alpha,\beta
,\gamma\right)  $ defined by means of the generating function \cite{A unified approach to generalized Stirling numbers}
\[
\sum\limits_{n=0}^{\infty}S_{n}\left(  x;\alpha,\beta,r\right)  \frac{t^{n}%
}{n!}=(1+\alpha t)^{\frac{r}{\alpha}}\exp\left(  \frac{x}{\beta}\left[
(1+\alpha t)^{\frac{\beta}{\alpha}}-1\right]  \right)  .
\]
These polynomials have the explicit expression \cite{A unified approach to generalized Stirling numbers}%
\begin{equation}
S_{n}\left(  x;\alpha,\beta,r\right)  =\sum_{k=0}^{n}S_{2}(n,k,\alpha
,\beta,r)x^{k}, \label{3.3}%
\end{equation}
and extension of Spivey's Bell number formula to $S_{n}\left(  x;\alpha
,\beta,\gamma\right)  $ \cite{LC,X} as
\[
S_{n+m}\left(  x;\alpha,\beta,r\right)  =\sum\limits_{k=0}^{n}\sum
\limits_{j=0}^{m}\binom{n}{k}S_{2}(n,k,\alpha,\beta,r)\left(  j\beta
-m\alpha\mid\alpha\right)  _{n-k}S_{k}\left(  x;\alpha,\beta,r\right)  x^{j}.
\]
Using this formula and the well-known identity%
\[
\sum\limits_{n=0}^{\infty}\left(  \beta\mid\alpha\right)  _{n}\frac{t^{n}}%
{n!}=(1+\alpha t)^{\frac{\beta}{\alpha}}%
\]

we have the following generating function for generalized exponential polynomials: 

\begin{lemma}
For all non-negative integer $m$%
\fontsize{10}{1}
\begin{equation}
\sum\limits_{n=0}^{\infty}S_{n+m}\left(  x;\alpha,\beta,r\right)  \frac{t^{n}%
}{n!}=(1+\alpha t)^{\frac{r-m\alpha}{\alpha}}\exp\left(  \frac{x}{\beta
}\left[  (1+\alpha t)^{\frac{\beta}{\alpha}}-1\right]  \right)  S_{m}\left(
x(1+\alpha t)^{\frac{\beta}{\alpha}};\alpha,\beta,r\right)  . \label{3.4}%
\end{equation}\normalsize
\end{lemma}
We also need the following lemma for the proof of main theorem of this
section. It is worth noting that this lemma is the analogue of Theorem 1 in
\cite{Higher order geometric polynomials}.
\begin{lemma}
For all non-negative integers $n$%
\begin{equation}
A_{n}^{\lambda,x}(\alpha,\beta,r)=\frac{\left(  -1\right)  ^{n}}{\Gamma\left(
\lambda\right)  }%
{\int\limits_{0}^{\infty}}
z^{\lambda-1}S_{n}\left(-\beta xz;\alpha,-\beta,-r\right)  e^{-z}dz,
\label{3.5}%
\end{equation}
where $\Gamma\left(\lambda\right)  $ is the well-known gamma function
\cite{AS}.
\end{lemma}

\begin{proof}
We have from (\ref{3.3}) that%
\[
\left(  -1\right)  ^{n}S_{n}\left(  -\beta xz;\alpha,-\beta,-r\right)
=\left(  -1\right)  ^{n}\sum_{k=0}^{n}S_{2}(n,k,\alpha,-\beta,-r)\left(
-\beta xz\right)  ^{k}.
\]
Then multiplying both sides of the above equation with $z^{\lambda-1}e^{-z}$,
integrating it with respect to $z$ from zero to infinity and using the
well-known identity of the gamma function%
$
\frac{\Gamma\left(  \lambda+k\right)  }{\Gamma\left(  \lambda\right)
}=\left(  \lambda\right)  ^{\bar{k}}=\binom{k+\lambda-1}{k}k!,\text{ }%
\lambda,k\in%
\mathbb{N}
,
$
we arrive at%

$
{\displaystyle\int\limits_{0}^{\infty}}
z^{\lambda-1}S_{n}\left(  -\beta xz;\alpha,-\beta,-r\right)  e^{-z}dz=\left(
-1\right)  ^{n}\Gamma\left(  \lambda\right)  A_{n}^{\lambda,x}(\alpha
,\beta,r),
$
which is the desired equation. Here, $\left(  \lambda\right)  ^{\bar{k}}$ is
the rising factorial function defined by\\ $\left(  \lambda\right)^{\bar{k}%
}=\lambda\left(  \lambda+1\right)  \cdots\left(  \lambda+k-1\right)  $ with
$\left(\lambda\right)^{\bar{0}}=1.$
\end{proof}
Then we have the main theorem of this section:

\begin{theorem}
For all non-negative integer $m$%
\fontsize{10}{1}
\[
A_{n}^{\lambda+m,-x-1}(\alpha,-\beta,\gamma)
\\=\frac{\left(  -1\right)  ^{m}%
}{\left(\lambda\right)^{\bar{m}}\left(\beta x\right)  ^{m}}%
\sum\limits_{k=0}^{m}(-1)^{k}S_{1}(m,k,\alpha,-\beta,-\gamma+m\alpha
-\lambda\beta)A_{n+k}^{\lambda,x}(\alpha,\beta,\gamma-m\alpha+\lambda\beta).
\]
\end{theorem}

\begin{proof}
From (\ref{3.5}) and (\ref{3.4}), we have%
\begin{align*}\fontsize{7}{1}
&  \sum\limits_{n=0}^{\infty}A_{n+m}^{\lambda,x}(\alpha,\beta,r)\frac{t^{n}%
}{n!}\\
&  \quad=\frac{\left(  -1\right)  ^{m}}{\Gamma\left(  \lambda\right)  }%
{\displaystyle\int\limits_{0}^{\infty}}
z^{\lambda-1}\left[  \sum\limits_{n=0}^{\infty}\frac{\left(  -t\right)  ^{n}%
}{n!}S_{n}\left(  -\beta xz;\alpha,-\beta,-r\right)  \right]  e^{-z}dz\\
&  \quad=\frac{\left(  -1\right)  ^{m}(1-\alpha t)^{\frac{-r-m\alpha}{\alpha}%
}}{\Gamma\left(  \lambda\right)  }%
{\displaystyle\int\limits_{0}^{\infty}}
z^{\lambda-1}\exp\left(  -z(1-x\left[  (1-\alpha t)^{\frac{-\beta}{\alpha}%
}-1\right]  )\right)\\&\qquad\times  S_{m}\left(  -\beta xz(1-\alpha t)^{\frac{-\beta}%
{\alpha}};\alpha,-\beta,-r\right)  dz\\
&  \quad=\frac{\left(  -1\right)  ^{m}(1-\alpha t)^{\frac{-r-m\alpha}{\alpha}%
}}{\Gamma\left(  \lambda\right)  }\sum\limits_{k=0}^{m}S_{2}(m,k,\alpha
,-\beta,-r)\left(  -\beta x(1-\alpha t)^{\frac{-\beta}{\alpha}}\right)  ^{k}%
\\&\qquad\times
{\displaystyle\int\limits_{0}^{\infty}}
z^{\lambda-1}\exp\left(  -z(1-x\left[  (1-\alpha t)^{\frac{-\beta}{\alpha}%
}-1\right]  )\right)  dz\\
&  \quad=\left(  -1\right)  ^{m}\sum\limits_{k=0}^{m}S_{2}(m,k,\alpha
,-\beta,-r)\binom{k+\lambda-1}{k}k!\left(  -\beta x\right)  ^{k}\frac
{1}{\left(  1-x[(1-\alpha t)^{\frac{-\beta}{\alpha}}-1]\right)  ^{\lambda+k}%
}\\&\qquad\times(1-\alpha t)^{\frac{-r-m\alpha-\beta k}{\alpha}}\\
&  \quad=\left(  -1\right)  ^{m}\sum\limits_{k=0}^{m}S_{2}(m,k,\alpha
,-\beta,-r)\binom{k+\lambda-1}{k}k!\left(  -\beta x\right)  ^{k}%
\sum\limits_{n=0}^{\infty}A_{n}^{\lambda+k,x}(\alpha,\beta,r+m\alpha
+k\beta)\frac{t^{n}}{n!}.
\end{align*}
Comparing the coefficients of the above equation gives\fontsize{11}{1}
$$
\left(  -1\right)  ^{m}A_{n+m}^{\lambda,x}(\alpha,\beta,r)=\sum\limits_{k=0}%
^{m}S_{2}(m,k,\alpha,-\beta,-r)\binom{k+\lambda-1}{k}k!\left(  -\beta
x\right)  ^{k}A_{n}^{\lambda+k,x}(\alpha,\beta,r+m\alpha+k\beta).
$$
\newpage

\newpage
Replacing $r$ with $\gamma-m\alpha+\lambda\beta,$ and using (\ref{3.7}), the
above equation can be written as%
\begin{align*}
&  \left(  -1\right)  ^{m}A_{n+m}^{\lambda,x}(\alpha,\beta,\gamma
-m\alpha+\lambda\beta)\\
&  \quad=\sum\limits_{k=0}^{m}S_{2}(m,k,\alpha,-\beta,-\gamma+m\alpha
-\lambda\beta)\binom{k+\lambda-1}{k}k!\left(  -\beta x\right)  ^{k}%
A_{n}^{\lambda+k,-x-1}(\alpha,-\beta,\gamma).
\end{align*}
Finally, using the generalized Stirling transform \cite[Eq. (4)]{A unified approach to generalized Stirling numbers}%
\[
f_{n}=\sum_{k=0}^{n}S_{2}(n,k,\alpha,\beta,\gamma)g_{k}\Leftrightarrow
g_{n}=\sum_{k=0}^{n}S_{1}(n,k,\alpha,\beta,\gamma)f_{k}%
\]
gives the desired equation.
\end{proof}
It is good to note that this theorem is counterpart of Theorem 2 of
\cite{barred:2013} and Theorem 3.7 of \cite{KC}.

Setting $x=-1/2$ in (\ref{equation:11}), we have%
\begin{equation}
\left[  \frac{2}{(1-\alpha t)^{\frac{-\beta}{\alpha}}+1}\right]  ^{\lambda
}(1-\alpha t)^{\frac{-\gamma}{\alpha}}=\sum\limits_{n=0}^{\infty}%
A_{n}^{\lambda,-1/2}(\alpha,\beta,\gamma)\frac{t^{n}}{n!}.\label{3.6}%
\end{equation}
Since the higher order degenerate Euler polynomials are defined by \cite{L},
\[
\left[  \frac{2}{(1+\alpha t)^{\frac{1}{\alpha}}+1}\right]  ^{\lambda
}(1+\alpha t)^{\frac{x}{\alpha}}=\sum\limits_{n=0}^{\infty}\mathcal{E}%
_{n}^{\left(  \lambda\right)  }(\alpha,x)\frac{t^{n}}{n!}%
\]
we define \textit{higher order generalised Euler polynomials with generalized Stirling
numbers} by the following generating function:
\begin{equation}
\left[  \frac{2}{(1+\alpha t)^{\frac{\beta}{\alpha}}+1}\right]  ^{\lambda
}(1+\alpha t)^{\frac{x}{\alpha}}=\sum\limits_{n=0}^{\infty}\mathcal{E}%
_{n}^{\left(  \lambda\right)  }(\alpha,\beta,x)\frac{t^{n}}{n!}.\label{3.11}%
\end{equation}
Using (\ref{3.6}) and (\ref{3.11}), we have%
\begin{equation}
\mathcal{E}_{n}^{\left(  \lambda\right)  }(\alpha,\beta,\gamma)=\left(
-1\right)  ^{n}A_{n}^{\lambda,-1/2}(\alpha,-\beta,-\gamma)=A_{n}%
^{\lambda,-1/2}(-\alpha,\beta,\gamma).\label{3.10}%
\end{equation}
Therefore, we may arrive at several formulas for higher order generalised Euler
polynomials with generalized Stirling numbers. For instance, we have the
explicit formulas%

\begin{align*}
\mathcal{E}_{n}^{\left(  \lambda\right)  }(\alpha,\beta,\gamma) &
=\sum\limits_{k=0}^{n}S(n,k,\alpha,\beta,\gamma)\binom{k+\lambda-1}{k}%
k!\frac{\left(  -\beta\right)  ^{k}}{2^{k}},\\
\mathcal{E}_{n}^{\left(  \lambda\right)  }(\alpha,\beta,\gamma) &
=\sum\limits_{k=0}^{n}S(n,k,\alpha,-\beta,\gamma-\beta\lambda)\binom
{k+\lambda-1}{k}k!\frac{\beta^{k}}{2^{k}},
\end{align*}

\newpage

recurrence relations
\begin{align*}\fontsize{7}{1}
\mathcal{E}_{n}^{\left(  \lambda+1\right)  }(\alpha,\beta,\gamma) 
&=2\mathcal{E}_{n}^{\left(  \lambda\right)  }(\alpha,\beta,\gamma
)-\mathcal{E}_{n}^{\left(  \lambda\right)  }(\alpha,\beta,\gamma+\beta),\\
\mathcal{E}_{n+1}^{\left(  \lambda\right)  }(\alpha,\beta,\gamma)& =\left(
\gamma-\lambda\beta\right)  \mathcal{E}_{n}^{\left(  \lambda\right)  }%
(\alpha,\beta,\gamma-\alpha)+\lambda\beta\mathcal{E}_{n}^{\left(
\lambda\right)  }(\alpha,\beta,\gamma-\alpha)\\&\quad-\left(  \lambda\beta/2\right)
\mathcal{E}_{n}^{\left(  \lambda\right)  }(\alpha,\beta,\gamma+\beta
-\alpha),\\
\mathcal{E}_{n}^{\lambda+m}(\alpha,\beta,-\gamma)   &=\frac{2^{m}}{\left(
\lambda\right)  ^{\bar{m}}\left(  \beta\right)  ^{m}}\sum\limits_{k=0}%
^{m}S_{1}(m,k,\alpha,-\beta,m\alpha-\lambda\beta-\gamma)\mathcal{E}%
_{n+k}^{\lambda}(\alpha,-\beta,m\alpha-\lambda\beta-\gamma),
\end{align*}
and convolution formulas%

\begin{align*}
\noindent\beta\left(\lambda_{1}+\lambda_{2}\right)&\sum_{k=0}^{n}\binom{n}%
{k}\mathcal{E}_{k}^{\left(\lambda_{1}\right)  }(\alpha,\beta,\alpha
+\gamma_{1}-\beta)\mathcal{E}_{n-k}^{\left(  \lambda_{2}\right)  }%
(\alpha,\beta,\gamma_{2})\\
&=2\left(  \gamma_{1}+\gamma_{2}\right)  \mathcal{E}_{n}^{\left(
\lambda_{1}+\lambda_{2}\right)  }(\alpha,\beta,\gamma_{1}+\gamma_{2}%
-\alpha)-2\mathcal{E}_{n+1}^{\left(  \lambda_{1}+\lambda_{2}\right)  }%
(\alpha,\beta,\gamma_{1}+\gamma_{2}),\\
  \sum_{k=0}^{n}\binom{n}{k}\mathcal{E}_{k}^{\left(  \lambda_{1}\right)
}(\alpha&,\beta,\alpha+\gamma_{1}-\beta)\mathcal{E}_{n-k}^{\left(  \lambda
_{2}\right)  }(\alpha,\beta,\gamma_{2})=\mathcal{E}_{n}^{\left(  \lambda
_{1}+\lambda_{2}\right)  }(\alpha,\beta,\alpha+\gamma_{1}+\gamma_{2}-\beta),\\
 \sum_{k=0}^{n}\binom{n}{k}\left(-1\right)  ^{n-k}&\mathcal{E}_{k}^{\left(
\lambda_{1}\right)  }(-\alpha,\beta,\alpha+\beta-\gamma_{1})\mathcal{E}%
_{n-k}^{\left(  \lambda_{2}\right)  }(\alpha,\beta,\gamma_{2}+\beta
\lambda)\\&=\mathcal{E}_{n}^{\left(  \lambda_{1}+\lambda_{2}\right)  }%
(-\alpha,\beta,\alpha+\beta-\gamma_{1}-\gamma_{2}).
\end{align*}

\section{Asymptotic Analysis}
In this section, an asymptotic expansion for higher order generalized geometric polynomials of negative indices is obtained. As in [2], a known result by Hsu \cite{Leetsch C Hsu 1990} is used to derive the asymptotic expansion. The method of deriving the said asymptotic expansion is very similar to that in [2]. 
\smallskip\\
Let $\psi(t)=\sum^{\infty}_{n=0}a_nt^n$ be a formal power series over the complex field with $a_0=\psi(0)=1$. For every $j (0 \leq j \leq n)$, let $W(n,j)$ be equal to the following sum, 
\begin{equation}
\label{result1}
W(n,j)=\sum_{1^{k_1}2^{k_2}\cdots n^{k_n} \in \sigma (n,n-j)} \frac{a_1^{k_1}a_2^{k_2}\cdots a_n^{k_n}}{k_1!k_2! \cdots k_n!},
\end{equation}
where $\sigma (n,n-j)$ denotes the set of partitions of $n$ with $(n-j)$ parts. The following result in \cite{Leetsch C Hsu 1990} will be used. 
\smallskip\\
Let $[t^n](\psi (t))^\lambda$ denote the coefficient of $t^n$ in the power series expansion of $(\psi(t))^\lambda$. Then for a fixed positive integer $s$ and for large $\lambda$ and $n$ such that $n=o(\lambda^{\frac{1}{2}})(\lambda \to \infty)$, the following asymptotic formula holds, 
\begin{equation}
\label{result2}
\frac{1}{(\lambda)_n} [t^n](\lambda(t))^\lambda=\sum^{s}_{j=0} \frac{W(n,j)}{(\lambda-n+j)_j} + o\left(\frac{W(n,s)}{(\lambda-n+s)_s} \right)
\end{equation}
where $(\lambda)_j=\lambda(\lambda-1)\cdots(\lambda-(j-1))$. In particular, when $n$ is fixed, the remainder estimate is given by $O(\lambda^{-s-1})$. 
\smallskip\\
To derive the asymptotic formula, consider the generating function in (\ref{equation:11}),  
\begin{equation}
(1-\alpha t)^{-\frac{\gamma}{\alpha}}\left[\frac{1}{1-x[(1-\alpha t)^{-\frac{\beta}{\alpha}}-1]} \right]^{\lambda}=\sum^\infty_{n=0}A^{\lambda, x}_n (\alpha, \beta, \gamma) \frac{t^n}{n!}. 
\end{equation} 
Let 
\begin{equation}
\psi(t)=(1-\alpha t)^{-\frac{\gamma}{\alpha}}\left[\frac{1}{1-x[(1-\alpha t)^{-\frac{\beta}{\alpha}}-1]} \right]=\sum^\infty_{n=0}A^{1, x}_n (\alpha, \beta, \gamma) \frac{t^n}{n!}. 
\end{equation}
Then 
\begin{equation}
(\psi(t))^\lambda= \frac{(1-\alpha t)^{-\frac{\lambda \gamma}{\alpha}}}{(1-x[(1-\alpha t)^{-\frac{\beta}{\alpha}}-1])^\lambda}=\sum^\infty_{n=0}A^{\lambda, x}_n (\alpha, \beta, \lambda \gamma) \frac{t^n}{n!}. 
\end{equation}
By making use of \eqref{result2}, 
\begin{equation}\label{result3}
\frac{A^{\lambda,x}_n(\lambda, \beta, \lambda\gamma)}{(\lambda)_n(n!)}=\sum^{s}_{j=0} \frac{W(n,j)}{(\lambda-n+j)_j} + o\left(\frac{W(n,s)}{(\lambda-n+s)_s} \right),
\end{equation}
where $n=o(\lambda^{\frac{1}{2}})$ as $\lambda \to \infty$, $W(n,j)$ are given in \eqref{result1}, and
\begin{align*}
a_j&=[t^j](\psi(t))\\
&=[t^j]\frac{(1+\alpha t)^{-\frac{\gamma}{\alpha}}}{(1-x((1-\alpha t)^{-\frac{\beta}{\alpha}}-1))}\\
&=\frac{A^{1,x}_j(\alpha, \beta, \gamma)}{j!}.
\end{align*}
 Note that
 \begin{align*}
 \sum^{\infty}_{n=0} A^{1,x}_n(\alpha, \beta, \gamma) \frac{t^n}{n!}&=\frac{(1-\alpha t)^{-\frac{\gamma}{\alpha}}}{(1-x((1-\alpha t)^{-\frac{\beta}{\alpha}}-1))}\\
 &=(1-\alpha t)^{-\frac{\gamma}{\alpha}}(1-x((1-\alpha t)^{-\frac{\beta}{\alpha}}-1))^{-1}\\
 &=(1-\alpha t)^{-\frac{\gamma}{\alpha}} \sum^{\infty}_{k=0} x^k[(1-\alpha t)^{-\frac{\beta}{\alpha}}-1]^k\\
 &=(1-\alpha t)^{-\frac{\gamma}{\alpha}} \sum^{\infty}_{k=0} x^k \sum^k_{j=0} \binom{k}{j}(1-\alpha t)^{-\frac{\beta}{\alpha}j}(-1)^{k-j}\\
&=\sum^{\infty}_{k=0} x^k \sum^k_{j=0}\binom{k}{j}(-1)^{k-j}(1-\alpha t)^{-\frac{\beta}{\alpha}j-\frac{\gamma}{\alpha}}\\
&=\sum^{\infty}_{k=0} x^k \sum^k_{j=0}\binom{k}{j}(-1)^{k-j}(1-\alpha t)^{- \frac{(\beta j+\gamma)}{\alpha}}\\  
&=\sum^{\infty}_{k=0} x^k \sum^k_{j=0}\binom{k}{j}(-1)^{k-j}\sum^\infty_{i=0}\binom{- \frac{(\beta j+\gamma)}{\alpha}}{i}(-\alpha t)^i\\
&=\sum^\infty_{i=0}\sum^{\infty}_{k=0}\sum^k_{j=0}x^k \binom{k}{j}(-1)^{k-j+i}\binom{- \frac{(\beta j+\gamma)}{\alpha}}{i}\alpha^it^i.
\end{align*}
Comparing coefficients,
\begin{equation*}
\frac{A^{1,x}_n(\alpha, \beta, \gamma)}{n!}=\sum^{\infty}_{k=0}\sum^k_{j=0}x^k \binom{k}{j}(-1)^{k-j+n}\binom{- \frac{\beta j+\gamma}{\alpha}}{n}\alpha^n
\end{equation*}
\begin{align*}
A^{1,x}_n(\alpha, \beta, \gamma)&=n!\alpha^n \sum^{\infty}_{k=0}\sum^k_{j=0}x^k \binom{k}{j}(-1)^{k-j+n}\binom{- \frac{(\beta j+\gamma)}{\alpha}}{n}\\
&=n!\alpha^n \sum^{\infty}_{k=0}\left\{ \sum^k_{j=0} \binom{k}{j}(-1)^{k-j+n}\binom{- \frac{(\beta j+\gamma)}{\alpha}}{n} \right\} x^k\\
&=\alpha^n \sum^{\infty}_{k=0}\left\{ \sum^k_{j=0} \binom{k}{j}(-1)^{k-j+n}\left(- \frac{(\beta j+\gamma)}{\alpha}\right)_n \right\} x^k\\
&=\alpha^n \sum^{\infty}_{k=0}\left\{ \sum^k_{j=0} \binom{k}{j}(-1)^{k-j+n} \frac{1}{\alpha^n} (-(\beta j+\gamma)|\alpha)_n \right\} x^k,
\end{align*}
which can be written
\begin{equation}
\label{resu}
A^{1,x}_n(\alpha, \beta, \gamma)=\sum^{\infty}_{k=0}\left\{(-1)^n \sum^k_{j=0} \binom{k}{j}(-1)^{k-j}(-\beta j-\gamma)|\alpha)_n \right\} x^k,
\end{equation}
where $(-(\beta j+\gamma|\alpha))_n=-(\beta j+\gamma)(-(\beta j+\gamma)-\alpha)\cdots(-(\beta j+\gamma)-(n-1)\alpha)$.
The following result in \cite{A unified approach to generalized Stirling numbers} (also mentioned as Lemma 2.1 in \cite{A comb analysis of geo polynomials}) shall be used.
\begin{lem}
For real or complex $\alpha, \beta, \gamma$,
\begin{equation}
\label{lemma}
S(n,k;\alpha, \beta, \gamma)=\frac{1}{\beta^k k!} \sum^k_{s=0} (-1)^{k-s} \binom{k}{s} (\beta s+\gamma|\alpha)_n. 
\end{equation} 
\end{lem}
From \eqref{lemma}, 
\begin{equation*}
\sum^k_{s=0} (-1)^{k-s} \binom{k}{s} ((\beta s+\gamma)|\alpha)_n= \beta^ k(k!)S(n,k;\alpha, \beta, \gamma).
\end{equation*}
Consequently, \eqref{resu} becomes,
\begin{align}
A^{1,x}_n(\alpha, \beta, \gamma)&=\sum^n_{k=0}(-1)^{n+k} \beta^k (k!) S(n,k; \alpha, -\beta, -\gamma)x^k\\
&=B_n(\alpha, -\beta, -\gamma; x),
\end{align} 
where $B_n(\alpha, -\beta, -\gamma;x)$ is a variation of the generalized Bell polynomial. Moreover, $S(n+1,k; \alpha, -\beta, -\gamma)$ satisfy the recurrence relation (see \cite{A unified approach to generalized Stirling numbers} Theorem 1), 
\begin{equation}
\label{recurrence}
S(n+1,k; \alpha, -\beta, -\gamma)=S(n, k-1; \alpha, -\beta, -\gamma)+(k(-\beta)-n\alpha+(-\gamma))S(n,k; \alpha, -\beta, -\gamma)
\end{equation} 
with
\begin{equation}
S(n,0; \alpha, \beta, \gamma)=(\gamma|\alpha)_n, \;\;\;\;\; S(n,n; \alpha, \beta, \gamma)=1, 
\end{equation} 
where $(\gamma| \alpha)_0=1$, and $(\gamma| \alpha)_1=\gamma$. \\
\indent For $j=0,1,2$, the $W(n,j)$'s are given below (see \cite{A comb analysis of geo polynomials, Leetsch C Hsu 1990}): 
\begin{align*}
W(n,0)&=\frac{1}{n!}a^n_1\\
W(n,1)&=\frac{1}{(n-2)!}a^{n-2}_1a_2\\
W(n,2)&=\frac{1}{(n-3)!}a^{n-3}_1a_3 + \frac{1}{2!(n-4)!}a^{n-4}_1a^2_2.
\end{align*}
For $\gamma=0$, 
\begin{align*}
a_1=\frac{A^{1,x}_1(\alpha, \beta, 0)}{1!}&=\sum^1_{k=0}(-1)^{1+k}\beta^k (k!) S(1,k; \alpha, -\beta, 0)x^k\\
&=-\beta^k S(1,0; \alpha, -\beta, 0)+ \beta S(1,1; \alpha, -\beta, 0)x=\beta x,
\end{align*}
\begin{align*}
a_2=\frac{A^{\lambda,x}_2(\alpha, \beta, 0)}{2!}&=\frac{1}{2} \left(\sum^2_{k=0}(-1)^{2+k}\beta^k (k!) S(2,k; \alpha, -\beta, 0 \right)x^k\\
&=\frac{1}{2}\left(-\beta S(2,1; \alpha, -\beta, 0)x+ \beta^2(2!) S(2,2; \alpha, -\beta, 0)x^2 \right)\\
&=\frac{1}{2}\beta(\beta+\alpha)x+\beta^2x^2,
\end{align*}
where, using \eqref{recurrence}, 
\begin{align*}
S(2,1; \alpha, -\beta, 0)&=S(1,0; \alpha, -\beta, 0)+(-\beta-\alpha) S(1,1; \alpha, -\beta, 0)\\
&=-\beta-\alpha, 
\end{align*}
and $S(2,2;\alpha, -\beta, 0)=1$,
\begin{align*}
a_3&=\frac{A^{\lambda,x}_3(\alpha, \beta, 0)}{3!}=\frac{1}{6} \left(\sum^3_{k=0}(-1)^{3+k}\beta^k (k!) S(3,k; \alpha, -\beta, 0 \right)x^k\\
&=\frac{1}{6}\left(\beta S(3,1; \alpha,\beta, 0)x-\beta^2(2!) S(3,2; \alpha, -\beta, 0)x^2 +\beta^3(3!) S(3,3; \alpha, -\beta, 0)x^3 \right)\\
&=\frac{1}{6}\left(\beta(\beta +\alpha)(\beta+2\alpha)x+6\beta^2(\beta+\alpha)x^2+6\beta^3x^3 \right),
\end{align*}
where
\begin{align*}
S(3,1; \alpha, \beta, 0)&=S(2,0; \alpha, -\beta, 0)+(-\beta-2\alpha)S(2,1; \alpha, -\beta, 0)\\
&=(-\beta-2\alpha)(-\beta-\alpha)\\
&=(\beta+\alpha)(\beta+2\alpha),
\end{align*}
\begin{align*}
S(3,2; \alpha, -\beta, 0)&=S(2,1; \alpha, -\beta, 0)+(-2\beta-2\alpha)S(2,2; \alpha, -\beta, 0)\\
&=(-\beta-\alpha)+(-2\beta-2\alpha)\\
&=-3\beta-3\alpha,\\
S(3,3; \alpha, -\beta, 0)&=1.
\end{align*}
Thus, for $\gamma=0$,
\begin{align*}
W(n,0)&=\frac{1}{n!}a^n_1=\frac{1}{n!}(\beta x)^n,\\
W(n,1)&=\frac{1}{(n-2)!}a^{n-2}_1a_2\\
&=\frac{1}{(n-2)!}(\beta x)^{n-2} \left\{ 
\frac{1}{2}\beta(\beta+\alpha)x+\beta^2x^2\right\}\\
&=\frac{1}{2(n-2)!}(\beta x)^{n-2}\left\{\beta(\beta+\alpha)x+2\beta^2x^2\right\}
\end{align*}

\begin{align*}
W(n,2)&=\frac{1}{(n-3)!}(\beta x)^{n-3}\left(\frac{1}{6}\right)\left(\beta(\beta +\alpha)(\beta +2\alpha)x+6\beta^2(\beta+\alpha)x^2+6\beta^3x^3\right)\\
&+\frac{1}{2!(n-4)!}(\beta x)^{n-4}\left(\frac{1}{2}\right)^2\left(\beta(\beta+\alpha)x+2\beta^2 x^2\right)^2\\
&=\frac{1}{3!(n-3)!}\left\{(\beta x)^{n-3}\left(\beta(\beta +\alpha)(\beta +2\alpha)x+6\beta^2(\beta+\alpha)x^2+6\beta^3x^3\right)\right\}\\
&+\frac{1}{8(n-4)!}\left\{(\beta x)^{n-4}\left(\beta(\beta+\alpha)x+2\beta^2 x^2\right)^2\right\}
\end{align*}
Using \eqref{result3},
\begin{align*}
\frac{A^{\lambda,x}_n(\alpha, \beta, 0)}{n!}&\sim(\lambda)_n\left(\frac{1}{n!}\right)(\beta x)^n+\frac{(\lambda)_n}{\lambda-(n-1)}\cdot \left(\frac{1}{2(n-2)!}(\beta x)^{n-2}(\beta(\beta+\alpha)x+2\beta^2x^2)\right)\\
&+\frac{(\lambda)_n}{(\lambda-n+2)_2}\Big(\frac{1}{3!(n-3)!}(\beta x)^{n-3}(\beta(\beta +\alpha)(\beta +2\alpha)x+6\beta^2(\beta+\alpha)x^2\\&+6\beta^3x^3)
+\frac{1}{8(n-4)!}(\beta x)^{n-4}\left(\beta(\beta+\alpha)x+2\beta^2 x^2\right)^2\Big)\\
&\sim(\lambda)_n\beta^n x^n+\frac{1}{2}(\lambda)_{n-1}(n)(n-1)\beta^{n-2}x^{n-2}(\beta(\beta+\alpha)x+2\beta^2x^2)\\
&+(\lambda)_{n-2}\bigg\{\frac{(n)(n-1)(n-2)}{3!}\beta^{n-3}x^{n-3}(\beta(\beta +\alpha)(\beta +2\alpha)x+6\beta^2(\beta+\alpha)x^2\\&+6\beta^3x^3)
+\frac{n(n-1)(n-2)(n-3)}{8}\beta^{n-4}x^{n-4}\left(\beta(\beta+\alpha)x+2\beta^2 x^2\right)^2\bigg\},
\end{align*}
which can be written
\begin{align}
A^{\lambda,x}_n(\alpha, \beta, 0)&\sim(\lambda)_n\beta^n x^n+(\lambda)_{n-1}\frac{(n)_2}{2!}\beta^{n-2}x^{n-2}(\beta(\beta+\alpha)x+2\beta^2x^2)\nonumber\\
&+\frac{(\lambda)_{n-2}(n)_3}{3!}\beta^{n-3}x^{n-3}\left(\beta(\beta +\alpha)(\beta +2\alpha)x+6\beta^2(\beta+\alpha)x^2+6\beta^3x^3\right)\nonumber\\
&+\frac{(\lambda)_{n-2}(n)_4}{8}\beta^{n-4}x^{n-4}\left(\beta(\beta+\alpha)x+2\beta^2 x^2\right)^2.
\end{align}
For any $\gamma$,
\begin{align*}
a_1&=\frac{A^{1,x}_1(\alpha,\beta,\gamma)}{1!}=\sum_{k=0}^1(-1)^{1+k}\beta^k\ k!\ S(1,k;\alpha, -\beta, -\gamma)x^k\\
&=(-1)(-\gamma)+\beta\ 1!\ S(1,1;\alpha, -\beta, -\gamma)x^1=\gamma+\beta x\\
a_2&=\frac{A^{1,x}_2(\alpha,\beta,\gamma)}{2!}=\frac{1}{2}\sum_{k=0}^2(-1)^{2+k}\beta^k\ k!\ S(2,k;\alpha, -\beta, -\gamma)x^k\\
&=\frac{1}{2}(\gamma(\gamma+\alpha)+(-1)\beta\ 1!\ S(2,1;\alpha, -\beta, -\gamma)x\\&\quad+\beta^2\ 2!\ S(2,2;\alpha, -\beta, -\gamma)x^2)\\
&=\frac{1}{2}\left\{\gamma(\gamma+\alpha)-\beta(-2\gamma-\beta-\alpha)x+2\beta^2 x^2\right\}\\
&=\frac{1}{2}\left\{(\gamma|-\alpha)_2+\beta(\beta+2\gamma+\alpha)x+2\beta^2 x^2\right\}
\end{align*}
where
\begin{align*}
S(2,0;\alpha, -\beta, -\gamma)&=(-\gamma|\alpha)_2=-\gamma(-\gamma-\alpha)=\gamma(\gamma+\alpha)\\
S(2,1;\alpha, -\beta, -\gamma)&=S(1,0;\alpha, -\beta, -\gamma)+(-\beta-\alpha-\gamma)S(1,1;\alpha, -\beta, -\gamma)\\
&=(-\gamma|\alpha)_1-\beta-\alpha-\gamma\\
&=-\gamma-\beta-\alpha-\gamma\\
&=-2\gamma-\beta-\alpha
\end{align*}

\begin{align*}
a_3&=\frac{A^{1,x}_3(\alpha,\beta,\gamma)}{3!}=\frac{1}{6}\sum_{k=0}^3(-1)^{3+k}\beta^k\ k!\ S(3,k;\alpha, -\beta, -\gamma)x^k\\
&=\frac{1}{6}\Big((-1)\ S(3,0;\alpha,-\beta,-\gamma)+\beta \ S(3,1;\alpha,-\beta,-\gamma)x-\beta^2\ 2!\ S(3,2;\alpha,-\beta,-\gamma)x^2\\&\quad+\beta^3 \ 3!\ S(3,3;\alpha,-\beta,-\gamma)x^3\Big)\\
&=\frac{1}{6}\Big((\gamma|-\alpha)_3+\beta((-\gamma|\alpha)_2+(\beta+\gamma+2\alpha)(\beta+2\gamma+\alpha))x-2\beta^2(-3(\beta+\alpha+\gamma))x^2\\&\quad+6\beta^3 x^3\Big)\\
&=\frac{1}{6}\left\{(\gamma|-\alpha)_3+\beta\left[(-\gamma|\alpha)_2+(\beta+\gamma+2\alpha)(\beta+2\gamma+\alpha)\right]x+6\beta^2(\beta+\alpha+\gamma)x^2+6\beta^3 x^3\right\}
\end{align*}
where
\begin{align*}
S(3,0;\alpha,-\beta,-\gamma)&=(-\gamma|\alpha)_3=(-\gamma)(-\gamma-\alpha)(-\gamma-2\alpha)=-(\gamma|-\alpha)_3\\
S(3,1;\alpha,-\beta,-\gamma)&=S(2,0;\alpha,-\beta,-\gamma)+(-\beta-2\alpha-\gamma)S(2,1;\alpha,-\beta,-\gamma)\\
&=(-\gamma|\alpha)_2+(-\beta-2\alpha-\gamma)(-2\gamma-\beta-\alpha)\\
&=(-\gamma|\alpha)_2+(\beta+\gamma+2\alpha)(\beta+2\gamma+\alpha)\\
S(3,2;\alpha,-\beta,-\gamma)&=S(2,1;\alpha,-\beta,-\gamma)+(-2\beta-2\alpha-\gamma)S(2,2;\alpha,-\beta,-\gamma)\\
&=-2\gamma-\beta-\alpha-2\beta-2\alpha-\gamma\\
&=-3\gamma-3\beta-3\alpha\\
&=-3(\beta+\alpha+\gamma)
\end{align*}
Consequently,
\begin{align*}
W(n,0)&=\frac{1}{n!}a^n_1=\frac{1}{n!}(\gamma+\beta x)^n\\
W(n,1)&=\frac{1}{(n-2)!}a_1^{n-2}a_2=\frac{1}{(n-2)!}(\gamma+\beta x)^{n-2}\cdot \frac{1}{2}\left((\gamma|-\alpha)_2+\beta(2\gamma+\beta+\alpha)x+2\beta^2 x^2\right)\\
&=\frac{1}{2!(n-2)!}(\gamma+\beta x)^{n-2}\left((\gamma|-\alpha)_2+\beta(\beta+2\gamma+\alpha)x+2\beta^2 x^2\right)\\
W(n,2)&=\frac{1}{(n-3)!}a_1^{n-3}a_3+\frac{1}{2!(n-4)!}a_1^{n-4}a_2^2\\
&=\frac{1}{(n-3)!}(\gamma+\beta x)^{n-3}\cdot \frac{1}{6}\bigg\{(\gamma|-\alpha)_3+\beta((-\gamma|\alpha)_2+(\beta+\gamma+2\alpha)(\beta+2\gamma+\alpha))x\\&\quad+6\beta^2(\beta+\alpha+\gamma)x^2
+6\beta^3 x^3\bigg\}\\&\quad+\frac{1}{2!(n-4)!}(\gamma+\beta x)^{n-4}\left(\frac{1}{2}\left((\gamma|-\alpha)_2+\beta(\beta+2\gamma+\alpha)x+2\beta^2 x^2\right)\right)^2
\end{align*}
\begin{align*}
&=\frac{1}{3!(n-3)!}(\gamma+\beta x)^{n-3}\bigg\{(\gamma|-\alpha)_3+\beta((-\gamma|\alpha)_2+(\beta+\gamma+2\alpha)(\beta+2\gamma+\alpha))x\\&\quad+6\beta^2(\beta+\alpha+\gamma)x^2\quad+6\beta^3 x^3\bigg\}\\&\quad+\frac{3}{4!(n-4)!}(\gamma+\beta x)^{n-4}\left\{(\gamma|-\alpha)_2+\beta(\beta+2\gamma+\alpha)x+2\beta^2 x^2\right\}^2
\end{align*}
Using (2),
\begin{equation*}
\frac{A^{\lambda,x}_n(\alpha, \beta,\lambda\gamma)}{(\lambda)_n n!}\sim W(n,0)+\frac{W(n,1)}{(\lambda-n+1)_1}+\frac{W(n,2)}{(\lambda-n+2)_2}.
\end{equation*}
That is,
\begin{align*}
A^{\lambda,x}_n(\alpha, \beta,\lambda\gamma)&\sim (\lambda)_n n!\cdot \frac{1}{n!}(\gamma+\beta x)^n+(\lambda)_{n-1} n!\frac{1}{2!(n-2)!}(\gamma+\beta x)^{n-2}((\gamma|-\alpha)_2\\&+\beta(\beta+2\gamma+\alpha)x+2\beta^2 x^2)+(\lambda)_{n-2} n!\cdot \frac{1}{3!(n-3)!}(\gamma+\beta x)^{n-3}\bigg\{(\gamma|-\alpha)_3\\&+\beta((-\gamma|\alpha)_2+(\beta+\gamma+2\alpha)(\beta+2\gamma+\alpha))x+6\beta^2(\beta+\alpha+\gamma)x^2+6\beta^3 x^3\bigg\}\\&+(\lambda)_{n-2} n!\frac{3}{4!(n-4)!}(\gamma+\beta x)^{n-4}\Big\{(\gamma|-\alpha)_2+\beta(\beta+2\gamma+\alpha)x+2\beta^2 x^2\Big\}^2,
\end{align*}
which can be written,
\begin{align*}
A^{\lambda,x}_n(\alpha, \beta,\lambda\gamma)&\sim (\lambda)_n(\gamma+\beta x)^n+(\lambda)_{n-1}\binom{n}{2}(\gamma+\beta x)^{n-2}\Big\{(\gamma|-\alpha)_2+\beta(\beta+2\gamma+\alpha)x\\&+2\beta^2 x^2 \Big\}
+(\lambda)_{n-2}\binom{n}{3}(\gamma+\beta x)^{n-3}\bigg\{(\gamma|-\alpha)_3+\beta((-\gamma|\alpha)_2\\&+(\beta+\gamma+2\alpha)(\beta+2\gamma+\alpha))x+6\beta^2(\beta+\alpha+\gamma)x^2+6\beta^3 x^3 \bigg\}\\&+3(\lambda)_{n-2}\binom{n}{4}(\gamma+\beta x)^{n-4}\Big\{(\gamma|-\alpha)_2
+\beta(\beta+2\gamma+\alpha)x+2\beta^2 x^2\Big\}^2.
\end{align*}

\end{document}